\documentclass[a4paper,9pt,leqno,openright,oneside]{article}
\usepackage[latin1,utf8x]{inputenc}

\usepackage{amsthm,amssymb,amsmath,amsopn,amsfonts,pb-diagram,footmisc}
\usepackage{cite}
\usepackage{stmaryrd,calc,enumerate,mathrsfs,amscd,wasysym,footnote}
\usepackage{graphicx}
\usepackage{fancyhdr}
\usepackage{indentfirst}
\usepackage{geometry}
\usepackage{mathtools}
\usepackage{bbm}
\usepackage{float}

\theoremstyle{plain}
\newtheorem{theorem}{Theorem}[section]

\newtheorem{prop}[theorem]{Proposition}

\theoremstyle{definition} 
\newtheorem{defn}[theorem]{Definition}
\newtheorem{remark}[theorem]{Remark}

\newcommand*{\bigchi}{\mbox{\Large$\chi$}}

\DeclareMathOperator{\Hol}{Hol}

\title{On a class of shift-invariant subspaces of the Drury-Arveson space}
\author{{Nicola Arcozzi}, {Matteo Levi}}
\date{}

\begin{document}

\maketitle

\begin{abstract}
In the Drury-Arveson space, we consider the subspace of functions whose Taylor coefficients are supported in the complement of a set $Y\subset\mathbb{N}^d$ with the property that $Y+e_j\subset Y$ for all $j=1,\dots,d$. This is an easy example of shift-invariant subspace, which can be considered as a RKHS in is own right, with a kernel that can be explicitely calculated. Every such a space can be seen as an intersection of kernels of Hankel operators with explicit symbols. Finally, this is the right space on which Drury's inequality can be optimally adapted to a sub-family of the commuting and contractive operators originally considered by Drury.
\end{abstract}

\section{Introduction}
We begin by fixing some notation and delimiting the framework we work in. Let $H$ be an abstract Hilbert space and for $d\geq 2$ consider a $d$-tuple of operators $A=(A_1,\dots,A_d):H \to H^d$. It is not difficult to see that the formal adjoint operator $A^*:H^d\to H$  acts as follows
\begin{equation*}
A^*k=\sum_j A^*_j k_j ,\quad \mbox{for} \ k=(k_1,\dots,k_d)\in H^d.
\end{equation*}
Given a polynomial $Q$ in $d$ variables, say $Q(z)=\sum_{k} c_k z^k$, where $z=(z_1,\dots,z_d)$, $k\in \mathbb{N}^d$ and the sum is finite, we write $Q(A)$ for the operator from $H$ to itself given by
\begin{equation*}
Q(A)=\sum_kc_kA^k=\sum_kc_kA_1^{k_1}\dots A_d^{k_d}.
\end{equation*}
Following Drury, we will relate $A$ to an operator acting on a Hilbert space of holomorphic functions of several variables on the unit ball. We write $\mathbb{B}^d$ for the open unit ball $\lbrace z=(z_1,\dots,z_d)\in\mathbb{C}^d \ : \ |z|<1\rbrace$, where $|z|^2:=\sum_{j=1}^d |z_j|^2$. Assuming that multiplication by $z_j$ defines a bounded linear operator (and it does on the spaces we are dealing with), on such a space we can consider a very natural $d$-tuple of operators, namely the $d$-shift
\begin{equation*}
M_z=(M_1,\dots,M_d):H\to H^d,
\end{equation*} 
where $M_j:f(z)\mapsto z_jf(z)$.\par
\begin{defn}
The \textit{Drury-Arveson space} is the space $H_d$ of functions $f(z)=\sum_{n\in \mathbb{N}^d}a(n)z^n$ holomorphic on the unit ball $\mathbb{B}^d\subset\mathbb{C}^d$, such that
\begin{equation*}
\Vert f\Vert_{H_d}^2:=\sum_{n\in \mathbb{N}^d}|a(n)|^2\beta(n)^{-1}<\infty,
\end{equation*}
where the weight function $\beta:\mathbb{N}^d\to \mathbb{N}$ is given by $\beta(n)=|n|!/n!$.\par
\end{defn}
This space has a reproducing kernel. For $f\in H_d$ and $z\in \mathbb{B}^d$, we have
\begin{equation*}
f(z)=\sum_{n}a_nz^n=\sum_{n}a_n\frac{z^n}{\beta(n)}\beta(n)=\langle f,k_z\rangle_{H_d},
\end{equation*}
with $k_z(w)=\sum_n \beta(n)\overline{z}^n w^n$ for $w\in\mathbb{D}$.\par
The series can be explicitly calculated and we get
\begin{equation*}
k_z(w)=\sum_{n\in\mathbb{N^d}} \beta(n)\overline{z}^n w^n=\sum_{k\geq 0}\sum_{|n|=k} \binom{k}{n}\overline{z}^n w^n=\sum_{k\geq 0}\Big(\sum_{j=1}^d \overline{z_j} w_j\Big)^k=\sum_{k\geq 0}\Big(\overline{z}\cdot w\Big)^k=\frac{1}{1-\overline{z}\cdot w}.
\end{equation*}
This function space was first introduced by Drury in \cite{drury}, then further developed in \cite{arv}. See also \cite{shalit}. It naturally arises as the right space to consider when trying to generalize to tuples of commuting operators a notable result by Von Neumann, saying that for any linear contraction $A$ on a Hilbert space and any complex polinomial $Q$, it holds
\begin{equation*}
\Vert Q(A) \Vert\leq \Vert Q\Vert_{\mathcal{M}(H^2)},
\end{equation*}
where $\mathcal{M}(H^2)=H^{\infty}$ denotes the multiplier space of the Hardy space of the unit disc $H^2$.\par
In fact, Drury shows that for a $d$-tuples of operators $A=(A_1,\dots,A_d):H \to H^d$, $d\geq 2$, such that $[ A_i,A_j]=0$ and $\Vert A \Vert\leq 1$, it holds
\begin{equation*}
\Vert Q(A) \Vert\leq \Vert Q\Vert_{\mathcal{M}(H_d)}.
\end{equation*}
The map $T$ given by
\begin{equation*}
(Tg)(z):=\sum_{n\in \mathbb{N}^d} g(n)\beta(n)z^n,
\end{equation*}
defines an isometric isomorphism from $\ell^2(\mathbb{N}^d,\beta)$ to $H_d$. This correspondence in particular tells us that the shift operator on $\ell^2(\mathbb{N}^d,\beta)$, given by
\begin{equation*}
S_j g(n)=\bigchi_{\mathbb{N}^d+e_j}(n)g(n-e_j)\beta(n-e_j)\beta(n)^{-1},
\end{equation*}
and the multiplication operator  $M_j$ on $H_d$ are unitarily equivalent, i.e. it turns out that $M_jT=TS_j$ for all $j=1,\dots,d$.

\section{A class of shift invariant subspaces of $H_d$}
We are interested in considering subspaces of $H_d$ of functions having Taylor coefficients with a prescribed support. Given some subset $X$ of $\mathbb{N}^d$, we write $\ell^2(X,\beta)$ for the closed subspace of $\ell^2(\mathbb{N}^d,\beta)$ of functions supported in $X$. We say that a set $X\subseteq \mathbb{N}^d$ is \textit{monotone}, if its complement in $\mathbb{N}^d$ is shift invariant, namely 
\begin{equation}\label{property}
\mathbb{N}^d\setminus X+e_j\subset \mathbb{N}^d\setminus X \qquad \mbox{for all} \ j=1,\dots,d,
\end{equation}
where $\mathbb{N}^d\setminus X$ is the complement of $X$ in $\mathbb{N}^d$. In all what follows we always consider $X$ to be a monotone set.\par
Given $g\in\ell^2(\mathbb{N}^d\setminus X,\beta)$, for $n \in X$ we have $S_jg(n)=0$ since $n-e_j\in X$ as well. Therefore $\ell^2(\mathbb{N}^d\setminus X,\beta)$ is a shift-invariant subspace of $\ell^2(\mathbb{N}^d,\beta)$. To any such a set $X$, we can associate the space $H_d(X)$ of functions of $H_d$ whose Taylor coefficients vanish on $\mathbb{N}^d\setminus X$. Since $M_jT\ell^2(\mathbb{N}^d\setminus X,\beta)=TS_j\ell^2(\mathbb{N}^d\setminus X,\beta)$, it follows that $H_d(\mathbb{N}^d\setminus X)$ is a shift-invariant subspace of $H_d$.\par
We can construct compressions of tuples of operators to the subspaces associated to the monotone set $X$.\par
In particular, let $B_j=S_j^*$ denote the backwards shift operator on $\ell^2(\mathbb{N}^d,\beta)$, given by $B_jg(n)=g(n+e_j)$. 
We consider the $d$-tuple of operators
\begin{equation*}
B^X=(B_1^X,\dots, B_d^X):\ell^2(X,\beta)\to\ell^2(X,\beta)^d,
\end{equation*}
 where for each $j=1,\dots,d$,
\begin{equation*}
B_j^X=P_XB_j\big|_{\ell^2(X,\beta)},
\end{equation*}
being $P_X$ the orthogonal projection of $\ell^2(\mathbb{N}^d,\beta)$ onto $\ell^2(X,\beta)$. In other words, $B_j^X$ is the compression of the standard $j^{th}$-backwards shift operator $B_j$ to $\ell^2(X,\beta)$.\par
Observe that the adjoint of $B^X$ is a \textit{row contraction} from $\ell^2(X,\beta)^d$ to $\ell^2(X,\beta)$,
\begin{equation*}
(B^X)^*(g_1,\dots,g_d)=\sum_j (B_j^X)^*g_j.
\end{equation*}
In the same way, we write $M_z^X$ for the compressed $d$-tuple $(M^X_1,\dots,M^X_d)$, where
\begin{equation*}
M_j^X=P_XM_j\big|_{H_d(X)},
\end{equation*}
$P_X$ being in this context the orthogonal projection from $H_d$ onto $H_d(X)$.

\section{Hankel operators and shift invariant subspaces}

Shift-invariant subspaces for the Drury-Arveson space are characterized in \cite{sunkes}, where it is shown that they can be represented as intersections of countably many kernels of Hankel operators, to be defined shortly. See also the PhD thesis \cite{sunkes2016}.\par
Consider a Hilbert space $\mathcal{H}$ of holomorphic functions on the unit ball $\mathbb{B}^d$, such that functions holomorphic on $\overline{\mathbb{B}^d}$ are dense in it. The function $b\in \mathcal{H}$ is a symbol if there exists $C>0$ such that
\begin{equation*}
|\langle f g, b\rangle_{\mathcal{H}}|\leq C\Vert f\Vert_{\mathcal{H}}\Vert g\Vert_{\mathcal{H}} \qquad \mbox{for all} \ f,g \in \Hol(\overline{\mathbb{B}^d}).
\end{equation*}
Endowing the space $\overline{\mathcal{H}}:=\lbrace \bar{f}:f\in\mathcal{H}\rbrace$ with the inner product $\langle \bar{f},\bar{g}\rangle_{\overline{\mathcal{H}}}:=\langle g,f\rangle_{\mathcal{H}}$, we say that $H_b:\mathcal{H} \to \overline{\mathcal{H}}$ is a Hankel operator with symbol $b\in \mathcal{H}$ if there exists $C>0$ such that
\begin{equation*}
\langle H_b f, \bar{g}\rangle_{\overline{\mathcal{H}}}=\langle f g, b\rangle_{\mathcal{H}}\qquad \mbox{for} \ f,g \in \Hol(\overline{\mathbb{B}^d}).
\end{equation*}

On $H_d$, consider the Hankel operator with symbol $b(z)=z^m$, for some $m\in\mathbb{N}^d$. We have $f \in \ker H_b$ iff $\langle f g, b\rangle=0$ for all $g \in \Hol(\overline{\mathbb{B}^d})$. Since,
\begin{equation*}
\langle f g, b\rangle_{H_d}=\widehat{fg}(m)\beta(m)=\big(\sum_{n,k}\widehat{f}(k)\widehat{g}(n)z^{n+k}\big)^{\wedge}(m)\beta(m)=\beta(m)\sum_{k}\widehat{f}(k)\widehat{g}(m-k),
\end{equation*}
it follows that $f \in \ker H_b$ iff $\widehat{f}(k)=0$ for $k\leq m$, i.e. $\widehat{f}\equiv 0$ on the rectangle $R_m=\lbrace n \in \mathbb{N}^d: n_j\leq m_j \ \forall j\rbrace$. Hence, $f\in H_d(\mathbb{N}^d\setminus X)$ with $X=R_m$. This is the easiest example of shift-invariant subspace of the Drury-Arveson space with explicit symbol.\par
Actually, each set $X$ satisfying (\ref{property}) can be associated to a collection of Hankel symbols. Observe that $X$ is bounded if and only if for all $j$ there exists $n \in \mathbb{N}^d\setminus X$ such that $n\in\mathbb{N}e_j$. In such a case, $X$ is a finite union of rectangles, $X=\bigcup_{k=1,\dots, K}R_{m_k}$ and hence,
\begin{equation*}
H_d(\mathbb{N}^d\setminus X)=\bigcap_{k=1,\dots, K}\ker H_{z^{m_k}}.
\end{equation*}
If $X$ is unbounded, then for every $j$ such that $\mathbb{N}^d\setminus X\cap\mathbb{N}e_j=\emptyset$, we have an increasing sequence of rectangles covering the strip unbounded in the $j-th$ direction. Summing up, it follows that
\begin{equation*}
H_d(\mathbb{N}^d\setminus X)=\bigcap_{k=1}^{\infty}\ker H_{z^{m_k}}.
\end{equation*}

\section{Drury type inequality}
In the introduction we have defined polynomials valued on operators, $Q(A)$. The concept of operators being variables of functions can be properly extended. Following Nagy and Foias \cite{nagy}, given a contraction $A$ on a Hilbert space $H$ one can define the holomorphic functional calculus
\begin{equation*}
\varphi(A):=\sum_k c_k A^k,
\end{equation*}
whenever $\varphi\in \mathcal{A}:=\lbrace a(z)=\sum_k c_kz^k : \ a\in \Hol(\mathbb{D}), a \ \mbox{continuous on} \ \mathbb{\overline{D}}, (c_k)\in \ell^{\infty} \rbrace$.\par
Now, for any $\varphi \in \Hol(\mathbb{D})$, the function $\varphi_r(\cdot):=\varphi(r\cdot)$ is in the class $\mathcal{A}$ for $r\in (0,1)$. Moreover, if $\varphi\in H^{\infty}$, we have the uniform bound $|\varphi_r(z)|\leq\Vert \varphi\Vert_{\infty}$, for $z\in\mathbb{D}$, $0<r<1$.
Hence, for every $\varphi\in H^{\infty}$ it can be defined the functional calculus
\begin{equation*}
\varphi(A)=\lim_{r\to1^-}\varphi_r(A),
\end{equation*}
whenever the above limit exists in the strong operator topology, which is always the case when $A$ is a completely non-unitary contraction (see \cite{nagy}).\par
In particular, for $\varphi\in \mathcal{M}(H_d)\subset H^{\infty}$ and $A=M_z$, we can define the operator of multiplication by $\varphi$ via the functional calculus
\begin{equation}\label{limite}
M_{\varphi}=\varphi(M_z)=\lim_{r\to1^-}\varphi_r(M_z).
\end{equation}
This defines a bounded operator from $H_d$ to itself, and its adjoint is clearly given by $(M_{\varphi})^*=\lim_{r\to1^-}(\varphi_r(M_z))^*$.\par
We have the following version of Drury's inequality.
\begin{theorem}
Let $H$ be an abstract Hilbert space and $A=(A_1,\dots,A_d):H \to H^d$, $d\geq 2$ a $d$-tuple of operators such that
\begin{itemize}
\item[(i)] $A_iA_j=A_jA_i \qquad \mbox{for} \ i,j=1,\dots,d$.
\item[(ii)] $\Vert Ah\Vert_{H^d}\leq \Vert h\Vert_H \qquad \mbox{for all} \ h\in H$.
\end{itemize}
Let $X$ be the complement in $\mathbb{N}^d$ of the set $N:=\lbrace n \in \mathbb{N}^d : A^n=0\rbrace$. Then for every complex polynomial $Q$ of $d$ variables, we have
\begin{equation}\label{inequality2}
\Vert Q(A)\Vert\leq\Vert Q(B^X)\Vert\leq \inf\lbrace \Vert \varphi\Vert_{\mathcal{M}(H_d)}: \varphi\in \mathcal{M}(H_d), \varphi(M_z^X)=Q(M_z^X) \rbrace.
\end{equation}

\end{theorem}
\begin{proof}
For $N=\emptyset$ we have $X=\mathbb{N}^d$ and this is just Drury's theorem, while for $N=\mathbb{N}^d\setminus\lbrace 0\rbrace$, $A$ reduces to a $d$-tuple of zeros (we set $0^0$ to be the identity). So, we suppose that $N$ (and hence $X$) is a proper subspace of $\mathbb{N}^d$.\par
It is enough to show that the theorem is true when $(ii)$ is replaced by the stronger condition
\begin{itemize}
\item[(ii)'] $\Vert Ah\Vert_{H^d}\leq r\Vert h\Vert_H \qquad \mbox{for all} \ h\in H$,
\end{itemize}
where $r\in(0,1)$.\par
We write $\widetilde{H}(X)$ for the space $\ell^2(X, \check{H},\beta)$, where $\check{H}$ has the same underlying space as $H$ but a different norm, $\Vert h\Vert_{\check{H}}=\Vert Dh\Vert_{H}$, where $D$ is the defect operator of $A$, $D=\sqrt{I-A^*A}$, (see \cite{drury} for the details). Drury constructs an injective isometry $\theta:H\to\widetilde{H}(\mathbb{N}^d)$, $\theta h(n):=A^nh$, and shows that $\widetilde{B}^m\theta=\theta A^m$ for all $m\in\mathbb{N}^d$ (here $\widetilde{B}$ is the $d$-tuple of backshifts on $\widetilde{H}(\mathbb{N}^d)$).\par
We rephrase this in our setting. Let $\pi_X$ be the orthogonal projection of $\widetilde{H}(\mathbb{N}^d)$ onto $\widetilde{H}(X)$, $\widetilde{B}_j^X:=\pi_X\widetilde{B}_j|_{\widetilde{H}(X)}$ and $\psi:=\pi_X\circ\theta$.\par
Since $\theta$ is an isometry, it is easy to see that that
\begin{equation}\label{isometry}
\psi \quad \mbox{is an isometry} \iff \theta h =0 \quad \mbox{on} \ \mathbb{N}^d\setminus X \iff A^n=0 \quad \mbox{for} \ n \in \mathbb{N}^d\setminus X.
\end{equation}
 We have
\begin{equation*}
\psi A_j=\pi_X\widetilde{B}_j\theta, \qquad  \mbox{and} \qquad \widetilde{B}^X_j\psi=\pi_X\widetilde{B}_j|_{\widetilde{H}(X)}\pi_X\theta=\pi_X\widetilde{B}_j\pi_X\theta.
\end{equation*}
For $n\in X$ and $h \in H$,
\begin{equation*}
(\widetilde{B}_j-\widetilde{B}_j\pi_X)\theta h (n)=\theta h (n+e_j)-\pi_X\theta h (n+e_j)=
\begin{cases}
 0 & n+e_j\in X\\
\theta h (n+e_j) & n+e_j\not\in X
\end{cases}
\end{equation*}
which equals zero by (\ref{isometry}). It follows that
\begin{equation*}
\psi A^m=(\widetilde{B}^X)^m\psi \qquad \mbox{for all} \ m\in\mathbb{N}.
\end{equation*}
At this point, it is standard (for example follow \cite{drury}) that for every complex polynomial $Q$ we have,
\begin{equation}\label{inequality}
\Vert Q(A)\Vert\leq\Vert Q(B^X)\Vert=\Vert Q(M_z^X)\Vert.
\end{equation}
The equality above follows from the intertwining relation $M_j^XT=T(B_j^X)^*$, where the operator $T$ in our case is the isometric isomorphism from $\ell^2(X,\beta)$ to $H_d(X)$ given by $(Tg)(z):=\sum_{n\in X} g(n)\beta(n)z^n$.

For $f(z)=\sum_na_nz^n \in H_d(X)$, we have $M_j^Xf(z)=\sum_{n\in X\cap X+e_j}a_{n-e_j}z^n$, and so
\begin{equation*}
\Vert M_j^Xf\Vert_{H_d}^2=\sum_{n\in X\cap X-e_j}|a_n|^2\beta(n+e_j)^{-1}\leq \sum_{n\in X\cap X-e_j}|a_n|^2\beta(n)^{-1}\leq \Vert f\Vert_{H_d}.
\end{equation*}
Then, all polynomials are multipliers for $H_d$ and
\begin{equation}\label{equa}
\Vert Q(M_z^X)\Vert=\Vert P_XQ(M_z)\Vert\leq\Vert Q(M_z)\Vert=\Vert Q\Vert_{\mathcal{M}(H_d)}.
\end{equation}
Of course, there are in general many functions $\varphi$ such that $P_X\varphi(M_z)=P_XQ(M_z)$. In particular, let $\varphi$ be a multiplier of $H_d$ such that $\widehat{\varphi(n)}=\widehat{Q(n)}$ for $n\in X$. Then, for any $g\in H_d$ we have,
\begin{equation*}\
\begin{split}
\Vert P_X Q(M_z)g-P_X\varphi(M_z)g\Vert_{H_d}&\leq\Vert P_X (Q(M_z)-\varphi_r(M_z))g\Vert_{H_d}+\Vert P_X( \varphi(M_z)-\varphi_r(M_z))g\Vert_{H_d}\\
&\leq \Vert\sum_X \widehat{\varphi(n)}(1-r^{|n|})M_z^ng\Vert_{H_d}+\Vert \varphi(M_z)g-\varphi_r(M_z)g\Vert_{H_d}\\
&\leq\sum_X(1-r^{|n|})\widehat{\varphi(n)} \Vert M_z^ng\Vert_{H_d}+\Vert \varphi(M_z)g-\varphi_r(M_z)g\Vert_{H_d}.
\end{split}
\end{equation*}
The term on the right goes to zero as $r\to1^-$, so it follows $P_X Q(M_z)=P_X\varphi(M_z)$. Then, (\ref{equa}) can be generalized as follows
\begin{equation*}
\Vert Q(M_z^X)\Vert=\Vert P_X\varphi(M_z)\Vert\leq\Vert \varphi(M_z)\Vert=\Vert \varphi\Vert_{\mathcal{M}(H_d)},
\end{equation*}
for any $\varphi \in \mathcal{M}(H_d)$ such that $\widehat{\varphi(n)}=\widehat{Q(n)}$. We have then proved that,
\begin{equation*}
\Vert Q(A)\Vert\leq\Vert Q(B^X)\Vert\leq \inf\lbrace \Vert \varphi\Vert_{\mathcal{M}(H_d)}: \varphi\in \mathcal{M}(H_d), \varphi(M_z^X)=Q(M_z^X) \rbrace.
\end{equation*}
\end{proof}
\begin{remark}
Observe that the first inequality in the theorem is optimal if the backshift $d$-tuple $B^X$ satisfies (i) and (ii) and if $\lbrace n \in \mathbb{N}^d : (B^X)^n=0\rbrace$ equals $N$. It is clear that condition (ii) holds for $B^X$, for every choice of $X$. Also, $n \in N \iff n+m \in N$ for all $m\in\mathbb{N}^d$ and since $N=\mathbb{N}^d\setminus X$ this is equivalent as asking $f(n+m)=0$ for all $m\in\mathbb{N}^d$, $f\in \ell^2(X,\beta)$. But $f(n+m)=(B^X)^nf(m)$ and so $\lbrace n \in \mathbb{N}^d : (B^X)^n=0\rbrace=N$.\par
On the other hand, the commuting property (i) is not fulfilled on most sets $X$. Of course, if $X$ is chosen such that $\ell^2(X,\beta)$ is backshift-invariant, then $B^X=B\big|_{\ell^2(X,\beta)}$ and (i) and (ii) hold, see \cite{drury}. More in general, doing standard calculations it is not hard to see that $B^X$ satisfies (i) if and only if
\begin{equation}\label{commutativity}
n,n+e_i+e_j,n+e_i \in X\implies n+e_j\in X, \quad \mbox{for} \ i,j=1,\dots,d.
\end{equation}
This is a shape-condition on the set $X$, saying that it cannot have any subset with one of the following configurations

\begin{figure}[H]
\centering
\includegraphics[scale=0.5]{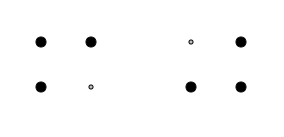}
\caption{Fat dots are elements of not permitted subsets of $X$.}
\end{figure}
It is clear that $X=N^C$ satisfies (\ref{commutativity}), since $n+e_j\in \mathbb{N}^d\setminus X$ for some $j$ would imply $n+e_j+e_i \in \mathbb{N}^d\setminus X$ for all $i=1,\dots,d$. It follows that the inequality in the theorem is optimal.
\end{remark}

\section{Further considerations}
We want to look closer at the inequality in (\ref{inequality2}). In particular, we are interested in understanding if it is an equality indeed. The reason to be optimistic in this sense comes from a theorem proved by Sarason in \cite{sar} (see also \cite[Theorem 3.1]{peller}) in the one-dimensional case, i.e. for the Hardy space. Let $K$ be a closed backshift-invariant subspace of the Hardy space $H^2$, and write $S_K$ for the compression of the shift operator to this subspace. Sarason proved the following.
\begin{theorem}
Let $T$ be an operator commuting with $S_K$. Then there exists a function $\varphi \in H^{\infty}$ such that $T=\varphi(S_K)$ and $\Vert T\Vert=\Vert \varphi\Vert_{H^{\infty}}$.
\end{theorem}
Now, on $H_1=H^2$ the operator $T=Q(M_z^X)$ clearly commutes with $M_z^X$, so there exists a function $\varphi\in\mathcal{M}(H^2)=H^{\infty}$, possibly different from the polynomial $Q$, such that $T=\varphi(M_z^X)$ and $\Vert T\Vert=\Vert \varphi\Vert_{\mathcal{M}_{H_d}}$. Then, (\ref{inequality2}) would become
\begin{equation*}
\Vert Q(A)\Vert\leq\Vert Q(B^X)\Vert=\Vert Q(M^X)\Vert=\Vert \varphi\Vert_{\mathcal{M}(H_d)}.
\end{equation*}
So we have equality in the case $d=1$. For higher dimensions, we have the following generalized \textit{commutant lifting theorem} (see \cite[Theorem 5.1]{Ball2001}).
\begin{theorem}\label{Ball}
Let $k(z,w)$ be a nondegenerate positive kernel on a domain $\Omega$ such that $1/k$ has 1 positive square. Let $H(k)$ be the associated RKHS. Suppose that $W\subset H(k)$ is a $\star$-invariant subspace and that $T$ is a bounded linear contraction from $W$ to itself such that
\begin{equation}
T^*M_{\varphi}^*|_{W}=M_{\varphi}^*T^*,
\end{equation}
for all $\varphi\in\mathcal{M}(H(k))$. Then, there exists a a multiplier $\psi\in\mathcal{M}(k)$ such that $\Vert (M_{\psi})\Vert\leq 1$ and $(M_{\psi})^*|_{W}=T^*$.
\end{theorem}
Asking that $1/k$ has 1 positive square means that the self adjoint matrix $\lbrace 1/k(z_i,z_j)\rbrace_{i,j=1}^N$ has exactly one positive eigenvalue, counted with multiplicity, for every finite set of disjoint points $\lbrace z_1,\cdots, z_N \rbrace\subset\mathbb{B}^d$. It is well known that the Drury-Arveson kernel has this property.\par
So, in order to apply the theorem, take $H_d$ as the RKHS and let $W=H^d(X)$. We have to show that $H_d(X)$ is $\star$\textit{-invariant}, i.e. that for every $\varphi\in \mathcal{M}(H_d)$ it holds $M^*_{\varphi}H^d(X)\subset H^d(X)$. Suppose that the multiplier function $\varphi$ has the power series expansion $\varphi(z)=\sum_n a_nz^n$. Then $\varphi_r(z)=\sum_n a_n(r)z^n$, where $a_n(r)=a_nr^{|n|}$. Using the fact that $(M_j^{n_j})^*=(M_j^*)^{n_j}$ and the uniform absolute convergence of the series, we get
\begin{equation}\label{formula}
(\varphi_r(M_z))^*=\Big(\sum_n a_n(r)M_z^n\Big)^*=\sum_n \overline{a_n(r)}\big(M_z^n\big)^*=\sum_n \overline{a_n(r)}\big((M_1^*)^{n_1},\dots,(M_d^*)^{n_d}\big).
\end{equation}
To prove the $\star$-invariance, thanks to (\ref{limite}) it is enough to show that $(\varphi_r(M_z))^*$ maps $H_d(X)$ in itself for all $r$, but this is immediate by (\ref{formula}), since $M_j^*$ does.\par
The operator $T=Q(M_z^X)=P_XM_Q$ maps continuously $H_d(X)$ to itself. Moreover we have,
\begin{equation*}
T^*M_{\varphi}^*|_{H_d(X)}=M_Q^*P_XM_{\varphi}^*|_{H_d(X)}=M_Q^*M_{\varphi}^*|_{H_d(X)}.
\end{equation*}
It follows that for $f\in H_d(X)$ it holds
\begin{equation*}
T^*M_{\varphi}^*f=M_Q^*M_{\varphi}^*f=(M_{\varphi}M_Q)^*f=(M_QM_{\varphi})^*f=M_{\varphi}^*M_Q^*f=M_{\varphi}^*M_Q^*P_Xf=M_{\varphi}^*T^*f.
\end{equation*}
Therefore, we have $T^*M_{\varphi}^*|_{H_d(X)}=M_{\varphi}^*T^*$.\par
Hence Theorem \ref{Ball} applies, and there exists a multiplier $\psi\in \mathcal{M}(H_d)$ such that and $(M_{\psi})^*|_{W}=(Q(M_z^X))^*$. In particular, it follows
\begin{equation}\label{suggestion}
\Vert Q(M_z^X)\Vert=\Vert M_{\psi} |_{H^d(X)}\Vert.
\end{equation}
\noindent \textbf{Question.}
Does this help in proving that equality holds in place of the second inequality in (\ref{inequality2}) for any dimension $d>1$?


\section{A closed formula for the reproducing kernel on slabs}
Let $X$ be some subset of $\mathbb{N}^d$ satisfying (\ref{property}). Clearly, the space $H_d(X)$ has a reproducing kernel $k^X(w,z)$ which is given by the orthogonal projection of the Drury-Arveson kernel onto $H_d(X)$, in the sense that
\begin{equation}\label{reproducing kernel}
k^X(w,z)=P_Xk(w,z)=P_Xk_z(w)=\sum_{n\in X} \beta(n)\overline{z}^n w^n.
\end{equation}


For some special choices of the set $X$ we are able to get a closed formula for the reproducing kernel in (\ref{reproducing kernel}). In particular, this can be done when $X$ is what we call a \textit{slab}, $\mathcal{S}_1=\lbrace n \in \mathbb{N}^d: n_1 =0,\dots,N_1\rbrace$.

\begin{prop}
For $X=\mathcal{S}_1$ it holds
\begin{equation}\label{slab}
k^{\mathcal{S}_1}(w,z)=\frac{1}{1-\overline{z}\cdot w}\Big(1-\frac{\overline{z}_1 w_1}{1-\overline{z}\cdot w+\overline{z}_1 w_1}\Big)^{N_1}.
\end{equation}
\end{prop}
\begin{proof}
Set $t=\overline{z}w$. As a first step, suppose that $d=2$. Using the fact that for $ j,k\in\mathbb{N}$ it holds
\begin{equation*}
\sum_{j=0}^{\infty}\binom{j+k}{j}x^j=\frac{1}{(1-x)^{k+1}},
\end{equation*}
we get
\begin{equation*}
\begin{split}
k^X(w,z)=\sum_{n\in X} \beta(n)\overline{z}^n w^n&=\sum_{n_1=0}^{N_1} \sum_{n_2=0}^{\infty}\binom{n_1+n_2}{n_2}t_1^{n_1}t_2^{n_2}=\sum_{n_1=0}^{N_1} t_1^{n_1}\frac{1}{(1-t_2)^{n_1+1}}\\
&=\frac{1}{1-t_2}\sum_{n_1=0}^{N_1}\Big(\frac{ t_1}{1-t_2}\Big)^{n_1}= \frac{1}{1-t_2}\frac{1-\Big(\frac{t_1}{1-t_2}\Big)^{N_1}}{1-\frac{t_1}{1-t_2}}\\
&=\frac{1-\Big(\frac{t_1}{1-t_2}\Big)^{N_1}}{1-t_1-t_2}=\frac{1}{1-\overline{z}\cdot w}\Big(1-\frac{\overline{z}_1 w_1}{1-\overline{z}\cdot w+\overline{z}_1 w_1}\Big)^{N_1}.
\end{split}
\end{equation*}
Now, suppose that (\ref{slab}) holds on $\mathbb{N}^{d-1}$. Again, suppose to re-order the basis $e_1,\dots,e_d$ so that $j=1$. On $\mathbb{N}^d$ we have
\begin{equation*}
\begin{split}
k^X(w,z)&=\sum_{n_1=0}^{N_1}\sum_{n_2=0}^{\infty}\dots\sum_{n_d=0}^{\infty}\binom{(|n|-n_d)+n_d}{n_d}\frac{(n_1+\dots+n_{d-1})!}{n_1!\dots n_{d-1}!}t_1^{n_1}\dots t_d^{n_d}\\
&=\sum_{n_1=0}^{N_1}\sum_{n_2=0}^{\infty}\dots\sum_{n_{d-1}=0}^{\infty}\frac{(n_1+\dots+n_{d-1})!}{n_1!\dots n_{d-1}!}t_1^{n_1}\dots t_{d-1}^{n_{d-1}}\frac{1}{(1-t_d)^{n_1+\dots+n_{d-1}+1}}\\
&=\frac{1}{1-t_d}\sum_{n_1=0}^{N_1}\sum_{n_2=0}^{\infty}\dots\sum_{n_{d-1}=0}^{\infty}\frac{(n_1+\dots+n_{d-1})!}{n_1!\dots n_{d-1}!}\Big(\frac{t_1}{(1-t_d)}\Big)^{n_1}+\dots+\Big(\frac{t_{d-1}}{(1-t_{d-1})}\Big)^{n_{d-1}}\\
&=\frac{1}{1-t_d}\frac{1}{1-\sum_{i=1}^{d-1}\frac{t_i}{(1-t_d)}}\Big(1-\frac{\frac{t_1}{1-t_d}}{1-\sum_{i=2}^{d-1}\frac{t_i}{(1-t_d)}}     \Big)^{N_1}\\
&=\frac{1}{1-\sum_{i=1}^d t_i}\Big(1-\frac{1}{1-\sum_{i=2}^d t_i}   \Big)^{N_1}=\frac{1}{1-\overline{z}\cdot w}\Big(1-\frac{\overline{z}_j w_j}{1-\overline{z}\cdot w+\overline{z}_1 w_1}\Big)^{N_1}.
\end{split}
\end{equation*}
\end{proof}

\bibliography{BiblioProc}{}
\bibliographystyle{plain}

\end{document}